\documentclass[12pt]{article}

\newtheorem{theo}{Theorem}
\newtheorem{prop}{Proposition}

\newtheorem{lemma}{Lemma}
\newtheorem{remark}{Remark}

\newcommand{\Ref}[1]{(\ref{#1})}

\newenvironment{proof}{\textbf{Proof. }}
{$\bigtriangleup$}

\newenvironment{proofTh1}{\textbf{Proof of Theorem 1. }}
{$\bigtriangleup$ }

\newenvironment{proofTh2}{\textbf{Proof of Theorem 2. }}
{$\bigtriangleup$ }

\newenvironment{eq}{\begin{equation}}{\end{equation}}

\newcommand{\si}{\sigma}

\newcommand{\al}{\alpha}
\newcommand{\be}{\beta}
\newcommand{\ga}{\gamma}
\newcommand{\la}{\lambda}
\newcommand{\de}{\delta}

\newcommand{\LA}{\langle}
\newcommand{\RA}{\rangle}
\newcommand{\ov}[1]{\overline{#1}}

\newcommand{\tr}{{\rm tr}}

\newcommand{\mdeg}{\mathop{\rm mdeg}}

\newcommand{\rank}{\mathop{\rm{rank }}}
\newcommand{\Spec}{\mathop{\rm{Spec }}}
\newcommand{\trdeg}{\mathop{\rm{tr.deg }}}
\newcommand{\St}{\mathop{\rm{St }}}
\newcommand{\Orb}{\mathop{\rm{O }}}

\begin{document}
 \title{The Invariant Ring of Triples of $3\times 3$ Matrices over a
Field of Arbitrary Characteristic}
  \author{A. A. Lopatin
  \\ Chair of Algebra,
 \\ Department of Mathematics, \\
 Omsk State University,\\ 55A Prospect Mira, Omsk 644077, Russia,
 \\e-mail: artem\underline{ }lopatin@yahoo.com  }
 \date{}
\maketitle
 \begin{abstract}
Let $R_{n,d}$ be the ring of invariants of $d$-tuples of $n\times n$
matrices under the simultaneous conjugation action of the general linear group.
A minimal generating system  and a homogeneous system of parameters
of $R_{3,3}$ are determined. Homogeneous systems of parameters
of $R_{3,2}$, $R_{4,2}$ are also pointed out.
\end{abstract}

\section{Introduction}
Let $K$ be an infinite field of the
characteristic $p$ ($p=0,2,3,\ldots$). Denote by
$M_{n,d}(K)=M_{n}(K) \oplus\cdots\oplus M_{n}(K)$ the space of
$d$-tuples of $n\times n$ matrices,
where $d\geq1$. The general linear group $GL_n(K)$
acts by simultaneous conjugation on $M_{n,d}(K)$: for $g\in
GL_n(K)$, $A_i\in M_n(K)$ $(i=\ov{1,d})$, we have
$$g(A_1,\ldots,A_d) = (gA_1g^{-1}, \ldots, gA_dg^{-1}).$$ The
coordinate ring of $M_{n,d}(K)$ is the
polynomial algebra $K_{n,d}=K[x_{ij}(r)\;|1\leq i,j\leq n,\;
r=\ov{1,d}]$, where $x_{ij}(r)$ denotes the function mapping
$(A_1,\ldots,A_d)\in M_{n,d}(K)$ to the $(i,j)$ entry of $A_r$. The action of
$GL_n(K)$ on $M_{n,d}(K)$ induces the action of $GL_n(K)$ on
$K_{n,d}$: $(g\cdot f)(A)=f(g^{-1}A)$, where $g\in GL_n(K)$, $f\in
K_{n,d}$, $A\in M_{n,d}$. Denote by
$$R_{n,d}=\{f\in K_{n,d}\;|\forall g\in GL_n(K): gf=f\}$$ the matrix algebra of invariants.
Let $X_r=(x_{ij}(r))_{1\leq i,j\leq n}$ be the generic matrices of
order $n$ $(r=\ov{1,d})$, and let $\si_k(A)$ be the coefficients
of the characteristic polynomial of $A\in M_n(K)$, i.e.,
$$\det(\la E-A)=\la^n-\si_1(A)\la^{n-1}+\cdots+(-1)^{n}\si_n(A).$$
Here $E$ stands for the identity matrix.
The algebra $R_{n,d}$ is generated by all elements of the form
$\si_k(X_{i_1}\cdots X_{i_s})$~\cite{Don}. The Procesi--Razmyslov
Theorem on relations of $R_{n,d}$ was extended to the case of
arbitrary $p$ in~\cite{Zub96}.

Let $N_0=\{0,1,2,\ldots\}$. The algebra $R_{n,d}$ possesses the
natural $N_0$-grading by degrees and $N_0^d$-grading by
multidegrees. If some $N_0^d$-homogeneous subset $G\subset R_{n,d}$
has the property that $G$ generates $R_{n,d}$
but any proper subset of $G$ does not generate $R_{n,d}$, then $G$ is
called a homogeneous minimal system of generators (shortly h.m.s.o.g.) of
$R_{n,d}$. A h.m.s.o.g. of $R_{2,d}$ was found in~\cite{Sibirskii}
when $p=0$, in~\cite{Procesi} when $p>2$, and in~\cite{DKZ} when
$p=2$. In~\cite{Dom} some upper and lower bounds on the highest degree of
elements of a h.m.s.o.g. of $R_{n,d}$ are pointed out for arbitrary
$p$. In~\cite{Abeasis} in case of $p=0$ the cardinality of a
h.m.s.o.g. of $R_{3,d}$ was calculated  for $d\leq 10$ by means a
computer, and shown a way how the set can be constructed for each $d$ by means
a computer. The least upper bound on degrees of elements of a
h.m.s.o.g. of $R_{3,d}$ was established in~\cite{Lopatin_Comm1} (except
for the case of $p=3$, $d=6k+1$, $k>0$, where error of estimation
of the least upper bound is not greater than $1$). In the given
paper we construct a h.m.s.o.g. of $R_{3,3}$ for arbitrary $p$.

By the Noether normalization
$R_{n,d}$ contains a {\it homogeneous} (with respect to $N_0$-grading)
{\it system of parameters} (shortly h.s.o.p), i.e. a set of
algebraically independent elements
generating a subalgebra over which $R_{n,d}$ is integral.
In case of $p=0$ Teranishi established h.s.o.p.-s of
$R_{3,2}$, $R_{4,2}$~\cite{Teranishi86} and of $R_{2,d}$ ($d\geq2$).
In~\cite{DKZ} h.s.o.p.-s  were established for $R_{2,d}$ for
$p>0$, $d\geq2$. The main result of the paper is the explicit description of
a h.s.o.p. for $R_{3,3}$ over a field of arbitrary characteristic. We also point out
h.s.o.p.-s for $R_{3,2}$, $R_{4,2}$ for arbitrary $p$.

\section{Auxiliary results}

Let $S$ be the free semigroup generated by letters
$\{x_1,x_2,\ldots\}$. For a word $u\in S$ we denote the degree of
$u$ by $\deg(u)$, the multidegree of $u$ by $\mdeg(u)$,
and the degree of $u$ with respect to the letter
$x_j$ by $\deg_{x_j}(u)$. All word are supposed to be non-empty unless
otherwise is stated.

Denote by $R_{n,d}^{+}$ the subalgebra generated by elements of
$R_{n,d}$ of positive degree. An element $r\in R_{n,d}$ is called {\it decomposable}
if it can be get in terms of elements whose degrees are less than that of $r$, i.e.
$r$ belongs to the ideal $(R_{n,d}^{+})^2$. Obviously, $\{r_i\}\in R_{n,d}$ is h.m.s.o.g.
iff $\{\ov{r_i}\}$ is a base of $\ov{R_{n,d}}=R_{n,d}/(R_{n,d}^{+})^2$.
If elements $r_1,r_2$ of $R_{n,d}$ are equal modulo $(R_{n,d}^{+})^2$,
we write $r_1\equiv r_2$. Denote by $K\LA x_1,\ldots,x_d\RA^{\#}$ the free associative
$K$-algebra without unity which is freely generated by $x_1,\ldots,x_d$.
Let ${\rm id}\{f_1,\ldots,f_s\}$ be the ideal generated by $f_1,\ldots,f_s$.
The problem of decomposability of an element of $R_{n,d}$ and the problem
of equality to zero of some element of $N_{n,d}=K\LA x_1,\ldots,x_d\RA^{\#} /
{\rm id}\{x^n\,|\,x\in K\LA x_1,\ldots,x_d\RA^{\#}\}$ are closely related
(see Lemma~\ref{lemmaLR} below).
Let $A_{n,d}$ be $K$-algebra without unity generated by generic matrices
$X_1,\ldots, X_d$. The homomorphism of algebras $\phi:A_{n,d}\to N_{n,d}$,
which maps $X_i$ in $x_i$, is defined correctly.

A word $w\in S$ is called {\it canonical} with respect to $x_i$,
if it has one of the following forms: $w_1$, $w_1x_iw_2$,
$w_1x_i^2w_2$, $w_1x_i^2ux_iw_2$, where words $w_1,w_2,u$ do not
contain $x_i$, words $w_1,w_2$ can be empty. If a word is
canonical with respect to all letters, then we call it {\it canonical}.
In~\cite{Lopatin_Comm1} the next lemmas are proved.

\begin{lemma}\label{lemmaLN}
1. Applying identities $x_iux_i=-x_i^2u-ux_i^2$, $x_iux_i^2=-x_i^2ux_i$  of
$N_{3,d}$, every non-zero word $w\in N_{3,d}$ can be represented as
a sum of canonical words which belong to the same homogeneous
component as $w$.

2. If $p\neq3$, then   $x_i^2ux_j^2=0$ holds in $N_{3,d}$.

3. The equality  $x_j^2 x_i^2 x_j x_i=-x_i^2 x_j^2 x_ix_j$ holds in $N_{3,d}$.

4. The inequality $x_1^2x_2^2x_1\neq0$ holds in $N_{3,d}$.

5. Let $p=3$ and for some identity $\sum\al_iu_i=0$ of $N_{3,d}$, $\al_i\in
K$, exists $k$ such that $\deg_{x_k}(u_i)=1,2$. Then after the substitution $x_k=1$
in $\sum\al_iu_i=0$ we get an identity of $N_{3,d}$.
\end{lemma}

\begin{lemma}\label{lemmaLR}
1. Let $G\in A_{n,d}$, $i=\ov{1,d}$. Then

a) if $G$ do not contain $X_i$ and $\tr(GX_i)$ is decomposable, then
$\phi(G)=0$;

b) if $\phi(G)=0$, then $\tr(GX_i)$ is decomposable.

2. Let $G\in A_{n,d}$ and $G$ do not contain $X_i$ for some
$i=\ov{1,d}$. Then decomposability of $\tr(GX_i^2)$ implies
$\phi(G)x_i+x_i\phi(G)=0$ in $N_{3,d}$.

3. Let $U\in A_{n,d}$ be a word. If $\tr(U)$ is indecomposable, then $\tr(U)$
can be represented in the form $\tr(U)\equiv\sum\al_i \tr(W_i)$, where
words $W_i\in A_{n,d}$ are canonical, $\mdeg(U)=\mdeg(W_i)$,
$\al_i\in K$.

4. The identity $\si_2(UV)\equiv\tr(U^2V^2)$, where
$U,V\in A_{n,d}$, holds in $R_{3,d}$.

5.  The explicit upper bound on degrees of elements of
a h.m.s.o.g. of $R_{3,3}$ is equal to $6$, if $p\neq3$, and is equal to $8$, if $p=3$.

6.  The explicit upper bound on degrees of elements of
a h.m.s.o.g. of $R_{3,d}$ $(d\geq2)$ is equal to $6$, if $p=0$.
 \end{lemma}

The result of substitution $X_1\rightarrow A_1,\ldots,X_d\rightarrow A_d$ in $r\in R_{n,d}$,
where $A_1,\ldots,A_d\in M_n(K)$ ($A_1,\ldots,A_d$ are some generic matrices, respectively),
denote by $r|_{A_1,\ldots,A_d}$. If we consider a subset of
$R_{n,d}$ instead of $r$, then we use the same notation.
If $Q\subset R_{n,d}$, $A_i\in M_n(K)$,
$i=\ov{1,d}$, and all elements of $Q|_{A_1,\ldots,A_d}$ are equal to zero,
then we write $Q|_{A_1,\ldots,A_d}=0$.

If $A,B\in M_n(K)$ are equivalent, i.e. there is $T\in GL_n(K)$
such that $A=TBT^{-1}$, then we write $A\sim B$. Let $\rank(A)$ be the
rank of a matrix $A$. Denote
$$J_1=\left(
\begin{array}{ccc}
0&1&0\\
0&0&0\\
0&0&0\\
\end{array}
\right),\,\, J_2=\left(
\begin{array}{ccc}
0&1&0\\
0&0&1\\
0&0&0\\
\end{array}
\right).
$$

Further $n$ is assumed to be $3$ unless otherwise is stated.

\section{A homogeneous  minimal system of generators}

Let $i,j,k=\ov{1,3}$ be pairwise different. Denote by $G_1$ the set
$$\begin{array}{l}
\tr(X_i);\, \\
\tr(X_iX_j),\, i<j;\,%
\si_2(X_i);\, \\
\tr(X_1X_2X_3),\,\tr(X_1X_3X_2);\,%
\tr(X_i^2X_j);\,%
\si_3(X_i);\, \\
\tr(X_i^2X_j^2),\, i<j;\,%
\tr(X_i^2X_jX_k);\,\\
\tr(X_i^2X_j^2X_k);\,%
\tr(X_i^2X_jX_iX_k),\, j<k.\\
\end{array}$$
Denote by $G_2$ the set
$$\begin{array}{l}
\tr(X^2_iX^2_jX_iX_j),\, i<j;\, %
\tr(X_i^2X_j^2X_iX_k);\,%
\tr(X_1^2X_2^2X_3^2).\\
\end{array}$$
Denote by $G_3$ the set
$$\begin{array}{l}
\tr(X^2_iX^2_jX_iX_j),\, i<j;\, %
\tr(X_i^2X_j^2X_iX_k);\,%
\tr(X_1^2X_2^2X_3^2),\,\tr(X_1^2X_3^2X_2^2);\,\\
\tr(X_iX_j^2X_k^2X_jX_k),\, j<k;\,%
\tr(X_i^2X_j^2X_iX_k^2),\, j<k;\,\\
\tr(X_i^2X_j^2X_k^2X_jX_k),\, j<k.\\
\end{array}$$

\begin{theo}

(i) If $p\neq3$, then $G_{\rm i}=G_1\cup G_2$ is a homogeneous minimal system
of generators of $R_{3,3}$.

\noindent(ii) If $p=3$, then $G_{\rm ii}=G_1\cup G_3$ is a homogeneous minimal
system of generators of $R_{3,3}$.
\end{theo}

\begin{remark} We have $|G_{\rm i}|=48$, $|G_{\rm ii}|=58$.
\end{remark}

In order to prove the Theorem we formulate the next Lemma.
\begin{lemma}\label{lemma_LI}
1. The elements $\tr(X_1X_2X_3)$, $\tr(X_1X_3X_2)$ are linearly
independent in $\ov{R_{3,3}}$.

\noindent 2. The elements $\tr(X_1^2X_2X_3)$, $\tr(X_1^2X_3X_2)$
are linearly independent in $\ov{R_{3,3}}$.

\noindent 3. The elements $\tr(X_1^2X_2^2X_3)$,
$\tr(X_2^2X_1^2X_3)$ are linearly independent in $\ov{R_{3,3}}$.

\noindent 4. If $p=3$, then the elements $\tr(X_1^2X_2^2X_3^2)$,
$\tr(X_1^2X_3^2X_2^2)$ are linearly independent in $\ov{R_{3,3}}$.
If $p\neq3$, then they are linearly dependent in $\ov{R_{3,3}}$.
\end{lemma}
\begin{proof} Items $1$, $2$ are consequences of item $3$, because
if $\sum_i\tr(U_iX_j)\equiv0$, where words $U_i\in A_{n,d}$ and $\deg_{X_j}(U_i)=0$,
then $\sum_i\tr(U_iX_j^2)\equiv0$.

3. Let $\al\tr(X_1^2X_2^2X_3)+\be\tr(X_2^2X_1^2X_3)\equiv0$, where
$\al\neq0$ or $\be\neq0$. Hence $\al x_1^2x_2^2+\be x_2^2x_1^2=0$
in $N_{3,d}$ by item~1 of Lemma~\ref{lemmaLR}. Since
$x_2^2x_1^2\neq0$ in $N_{3,d}$ (item~4 of Lemma~\ref{lemmaLN}), we
have $\al\neq0$. Thus $x_1^2x_2^2x_1=-(\be/\al)x_2^2x_1^3=0$ in
$N_{3,d}$, but $x_1^2x_2^2x_1\neq0$ in $N_{3,d}$ (item~4 of Lemma~\ref{lemmaLN});
a contradiction.

4. Let $p=3$ and $\al\tr(X_1^2X_2^2X_3^2) + \be
\tr(X_2^2X_1^2X_3^2) \equiv0$, where $\al\neq0$ or $\be\neq0$.
Hence $(\al x_1^2x_2^2 +\be x_2^2x_1^2)x_3 +x_3(\al x_1^2x_2^2+\be
x_2^2x_1^2)=0$ in $N_{3,d}$ (item~2 of Lemma~\ref{lemmaLR}).
Substitution $x_3=1$ gives $2\al x_1^2x_2^2+2\be x_2^2x_1^2=0$
(item~5 of Lemma~\ref{lemmaLN}), which was proved to be a
contradiction.

Let $p\neq3$. The identity $x_1^2x_2x_3^2=0$ in $N_{3,d}$
(item~2 of Lemma~\ref{lemmaLN}) implies
$\tr(X_1^2X_2X_3^2X_2)\equiv0$ (item~1 of Lemma~\ref{lemmaLR}). On
the other hand, $x_2x_3^2x_2=-x_2^2x_3^2-x_3^2x_2^2$ in
$N_{3,d}$ (item~1 of Lemma~\ref{lemmaLN}) implies
$\tr(X_1^2X_2X_3^2X_2)\equiv -\tr(X_1^2X_2^2X_3^2) -
\tr(X_1^2X_3^2X_2^2)$ (item~1 of Lemma~\ref{lemmaLR}).
\end{proof}

\begin{proofTh1} Lemmas~\ref{lemmaLN},~\ref{lemmaLR} imply
that $G_{\rm i}$ ($G_{\rm ii}$, respectively) generates
$R_{3,3}$ when $p\neq3$ ($p=3$, respectively).  These lemmas also
show that all elements of $G_{\rm i}$ (of $G_{\rm ii}$,
respectively) are  indecomposable, when $p\neq3$ ($p=3$, respectively).
Thus it is enough to prove that the elements of $G_{\rm i}$ (of
$G_{\rm ii}$, respectively) of the equal multidegree are linearly
independent in $\ov{R_{3,3}}$. The last follows from Lemma~\ref{lemma_LI}.
\end{proofTh1}

\smallskip
The next proposition follows easily from~\cite{Abeasis}.

\begin{prop}
Let $p=0$. Then the cardinality of a h.m.s.o.g. of $R_{3,d}$ equals to
$M_d=3d+5(_2^d)+24(_3^d)+51(_4^d)+47(_5^d)+15(_6^d)$, where $(_i^d)=0$ for
$i>d$.
\end{prop}
\begin{proof}Consider some h.m.s.o.g. of $R_{3,d}$.
Denote by $G_d$ its subset which consists of elements depending on $X_1,\ldots,X_d$.
It is easy to see, that
$$\bigcup_{k=1}^d\; \bigcup_{1\leq i_1<\cdots<i_k\leq d}G_k|_{X_{i_1},\ldots,X_{i_k}}$$
is a h.m.s.o.g. of $R_{3,d}$. Thus $M_d=\sum_{k=1}^d a_k(_k^d)$, where
$a_k=|G_k|$. The algebra $R_{3,d}$ is generated
by elements of the degree not greater than $6$
(for example, see item~$6$ of Lemma~\ref{lemmaLR}).  Hence
$a_i=0$ for $i>6$. In~\cite{Abeasis} it was shown that
$$\begin{array}{ccccccc}
d& 1 & 2 & 3 & 4 & 5 & 6\\
M_d& 3 & 11 & 48 & 189 & 607 & 1635\\
\end{array}.
$$
Thus we get the required.
\end{proof}

\section{A homogeneous system of parameters}
\begin{theo}
The set $P\subset R_{3,3}$
$$\begin{array}{l}
\si_k(X_i),\, i,k=\ov{1,3}, \\
\tr(X_1X_2),\,\tr(X_1X_3),\, \tr(X_2X_3), \\
\tr(X_1^2X_2)+\al_1\tr(X_2^2X_3)+\al_2\tr(X_3^2X_1), \\
\tr(X_1^2X_3)-\be_1\tr(X_3^2X_2),\,\tr(X_1^2X_3)-\be_2\tr(X_2^2X_1), \\
\tr(X_1X_2X_3)+\ga\tr(X_1X_3X_2),\\
\tr(X_1^2X_2^2),\, \tr(X_1^2X_3^2),\, \tr(X_2^2X_3^2),\\
\end{array}$$
 where $\al_1,\,\al_2,\,\be_1,\,\be_2,\,\ga$ are non-zero elements of $K$ and
 $\al_1+\be_1+\al_2\be_2\neq0$,
 is a homogeneous system of parameters of $R_{3,3}$.
\end{theo}

It is easy to see that if the statement of the Theorem is valid over the algebraic
closure of $K$, then it is valid over $K$. Therefore we can assume that $K$ is algebraically
closed.

We need the following Theorem, which was proved by
Hilbert for $p=0$~\cite{Hilbert}. In~\cite{Kraft} there is a proof of the Theorem
for $p=0$, and that proof is suited for arbitrary $p$.

{\it Let an algebraic group $G$ acts regularly on some affine variety $X$.
This action induces the action of $G$ on the coordinate ring $K[X]$ which
consists of regular mappings from $X$ into $K$. Let invariants
$I_1,\ldots,I_s\in K[X]^{G}$ have the property:
if $I_1(x)=\cdots=I_s(x)=0$, where $x\in X$, then for each homogeneous non-constant
invariant $I\in K[X]^G$ we have $I(x)=0$. Then the
ring of invariants $K[X]^G$ is integral over the subring generated by
$I_1,\ldots,I_s$.}

Denote by  $\trdeg (R_{n,d})$ the transcendence degree of $R_{n,d}$
$(n\geq 2)$, i.e. the cardinality of its h.m.s.o.g.
The next Lemma is a folklore statement but we give its proof for completeness.
\begin{lemma}\label{lemma_number}
$\trdeg (R_{n,d})=(d-1)n^2+1$, where $d\geq2$.
\end{lemma}
\begin{proof}  Recall the Theorem whose proof can be found, for example,
in~\cite{Humph}.

{\it If $\phi: X\to Y$ is a dominant
morphism of irreducible algebraic varieties, then there is an open non-empty
subspace $U\subset Y$ such that for all $y\in U$ we have
$\dim(\phi^{-1}(y)) = \dim(X) - \dim(Y)$.}

Let $\pi: M_{n,d}(K)\to M_{n,d}(K) / GL_n(K)=
\Spec(R_{n,d})$ be categorical quotient, where
$\Spec(R_{n,d})$ is the affine variety whose coordinate ring is isomorphic to $R_{n,d}$.
Denote by $F_d$ the free associative algebra generated freely by $f_1,\ldots, f_d$.
Each representation $\Psi: F_d\to M_{n}(K)$ correspond to some
point $(\Psi(f_1),\ldots,\Psi(f_{d}))\in M_{n,d}(K)$ and vice visa.
Denote by $W$ the set of points of
$M_{n,d}(K)$ which correspond to the simple representations.
It is known that if a point belongs to $W$, then its $GL_n(K)$-orbit is closed and
coincides with  the fiber of $\pi$ which contains it (see~\cite{Artin},~\cite{Procesi74}).

By Frobenius theorem representation $\Psi$ is simple iff there are
$g_1,\ldots,g_{n^2}\in F_d$ such that $\Psi(g_1),\ldots,\Psi(g_{n^2})$ are linearly
independent. Thus $\Psi$ is simple iff
$$\delta=\det\left(
\begin{array}{ccc}
\Psi(g_1)_1 & \cdots &  \Psi(g_1)_{n^2}\\
\vdots&&\vdots\\
\Psi(g_{n^2})_1 & \cdots & \Psi(g_{n^2})_{n^2}\\
\end{array}
\right)\neq0,$$
where $\Psi(g_i)_j$ stands for the $j^{\rm th}$ entry of matrix $\Psi(g_i)$. The number
$\de$ is equal to the value of some polynomial $h_{g_1,\ldots,g_{n^2}}\in K_{n,d}$, depending
on $g_1,\ldots,g_{n^2}$, calculated at the point
$x_{ij}(r)=\Psi(f_r)_{ij}$, where $\Psi(f_r)_{ij}$ is the $(i,j)$ entry of the matrix
$\Psi(f_r)$, $i,j=\ov{1,n}$, $r=\ov{1,d}$.  For
$h\in K_{n,d}$ denote $\ov{V}(h)=\{A\in M_{n,d}(K)| h(A)\neq0\}$.
We have
$$W=\bigcup\limits_{g_1,\ldots,g_{n^2}\in F_d} \ov{V}(h_{g_1,\ldots,g_{n^2}})$$
is an open set.

Apply the Theorem mentioned to $\pi$.
Since $M_{n,d}(K)$ is irreducible, there is  $x\in\pi^{-1}(U)\cap W$.
Thus there exists $y\in U$ such that $x\in \pi^{-1}(y)$. Hence
$\Orb(x)=\pi^{-1}(y)$, where $\Orb(x)$ denotes the orbit of $x$. We
have $\dim\Orb(x)=\dim GL_n(K)-\dim \St(x)$, where $\St(x)$ stands for the
stabilizer of $x$. Since $\dim GL_n(K)=n^2$, $\dim \St(x)=1$,
$\dim M_{n,d}=n^2d$ and $\dim \Spec(R_{n,d})=\trdeg(R_{n,d})$, the
Lemma is proved.
\end{proof}

\begin{lemma}\label{lemma0}
Let $A\in M_3(K)$ and $\si_k(A)=0$ ($k=\ov{1,3}$). Then
$\rank(A)\leq2$. Moreover,
$$\begin{array}{cccccc}
\rank(A)=0&{\rm iff}& A=0,&\qquad\qquad&\qquad\qquad&\qquad\qquad\qquad\qquad\\
\rank(A)=1&{\rm iff}& A\sim J_1,&&&\\
\rank(A)=2&{\rm iff}& A\sim J_2,&&&\\
\rank(A)\leq1&{\rm iff}&A^2=0.&&&
\end{array}$$
\end{lemma}
\begin{proof}
Lemma follows from the Theorem on Jordan form of matrix and the Cayley--Hamilton identity:
$A^3-\tr(A)A^2+\si_2(A)A-\det(A)E=0$, where $A\in M_3(K)$.
\end{proof}

\begin{lemma}\label{lemmaI}
Let $A=J_2$, $B\in M_3(K)$, $\rank(B)=2$, $\si_k(B)=0$
($k=\ov{1,3}$), $\tr(AB)=0$, $\tr(A^2B^2)=0$. Then one of the
following possibilities is valid:
\begin{enumerate}
\item[(1)] $B=\left(
\begin{array}{ccc}
0&t_1&t_2\\
0&0&t_3\\
0&0&0\\
\end{array}
\right)$.
\item[(2)] $B=TAT^{-1}$, $T=\left(
\begin{array}{ccc}
t_3t_4&t_1&t_2\\
t_3&t_4&t_5\\
0&1&0\\
\end{array}
\right)$, where $\det(T)\neq0$.
\item[(3)] $B=TAT^{-1}$, $T=\left(
\begin{array}{ccc}
t_1&t_2&t_3t_4\\
t_3&0&t_4\\
1&0&0\\
\end{array}
\right)$, where $\det(T)\neq0$.
\end{enumerate}
\end{lemma}
\begin{proof} Lemma~\ref{lemma0} implies $B=TJ_2T^{-1}$ for some $T\in GL_3(K)$.
Let $$L=\left(
\begin{array}{ccc}
a&b&c\\
0&a&b\\
0&0&a\\
\end{array}
\right),$$ where $a,b,c\in K$, $a\neq0$. We have
$TJ_2T^{-1}=(TL)J_2(TL)^{-1}$. There are $a,b,c$
$(a\neq0)$ such that $TL$ is equal to one of the following matrices:
$$\left(
\begin{array}{ccc}
\ast&\ast&\ast\\
\ast&\ast&\ast\\
0&0&1\\
\end{array}
\right),\quad \left(
\begin{array}{ccc}
\ast&\ast&\ast\\
\ast&\ast&\ast\\
0&1&0\\
\end{array}
\right),\quad \left(
\begin{array}{ccc}
\ast&\ast&\ast\\
\ast&\ast&\ast\\
1&0&0\\
\end{array}
\right).
$$

Substituting $TL$ for $T$ and considering the equations
$\tr(AB)=\tr(A^2B^2)=0$ we get the required.
\end{proof}

\begin{lemma}\label{lemmaII}
Let $A=J_1$, $\si_k(B)=0$ ($k=\ov{1,3}$), $\tr(AB)=0$,
$\rank(B)=1$. Then $AB=0$ or $BA=0$.
\end{lemma}
\begin{proof} Let $B=(b_{ij})_{1\leq i,j\leq3}$. The equation $\tr(AB)=0$ implies $b_{21}=0$.
Lemma~\ref{lemma0} states $B^2=0$. Thus we obtain
$b_{23}b_{31}=0$. Hence we have three alternatives: $b_{23}=0$,
$b_{31}\neq0$, or $b_{23}\neq0$, $b_{31}=0$, or $b_{23}=b_{31}=0$.
Performing direct computations and making appropriate
substitutions we get that one of the following possibilities is
valid:
\begin{enumerate}
\item[(1)] $B=\left(
\begin{array}{ccc}
b_3&b_2b_3/b_1&-b_3^2/b_1\\
0&0&0\\
b_1&b_2&-b_3\\
\end{array}
\right)$, where $b_1\neq0$.
\item[(2)] $B=\left(
\begin{array}{ccc}
0&b_1b_3/b_2&b_1\\
0&b_3&b_2\\
0&-b_3^2/b_2&-b_3\\
\end{array}
\right)$, where $b_2\neq0$.
\item[(3)] $B=\left(
\begin{array}{ccc}
0&b_1&b_2\\
0&0&0\\
0&b_3&0\\
\end{array}
\right)$, where $b_2=0$ or $b_3=0$.
\end{enumerate}
The Lemma is proved.
\end{proof}

The next Lemma was proved in~\cite{Teranishi86}.
\begin{lemma}\label{lemma1}
Let $A,B\in M_3(K)$, $\rank(A)=2$, $\si_k(A)=\si_k(B)=0$
($k=\ov{1,3}$), and $\tr(AB)=\tr(A^2B^2)=\tr(A^2B)=0$. Then
$\tr(AB^2)=0$. If also $A=J_2$, then $B$ is a strictly upper
triangular matrix.
\end{lemma}
\begin{proof}
Conjugating, we can suppose that $A=J_2$ and
$B=(b_{ij})_{1\leq i,j\leq3}$. The equation $\tr(A^2B)=0$ implies
$b_{31}=0$, and $\tr(AB)=b_{21}+b_{32}=0$,
$\tr(A^2B^2)=b_{21}b_{32}=0$ imply $b_{21}=b_{32}=0$. It follows
from $\si_k(B)=0$, $k=\ov{1,3}$, that $b_{11}=b_{22}=b_{33}=0$.
Thus $B$ is a strictly upper triangular matrix and $\tr(AB^2)=0$.
\end{proof}

\begin{lemma}\label{lemma3}
Let $Q\subset R_{3,d}$ be the set
$$\begin{array}{l}
\si_k(X_i),\; k=\ov{1,3},\; i=\ov{1,d},\quad \tr(X_i^2X_j),\;
1\leq i\neq j \leq d,\quad \\
\tr(X_i^2X_j^2),\; 1\leq i<j\leq d,\quad \tr(X_{\si(1)}\cdots
X_{\si(r)}),\; r=\ov{2,d},\; \si\in S_r.
\end{array}$$
If $A_1,\ldots,A_s\in M_3(K)$ are such that
$Q|_{A_1,\ldots,A_d}=0$, then for each $r\in R_{3,d}^{+}$ we have
$r|_{A_1,\ldots,A_d}=0$.
\end{lemma}
\begin{proof} If there is $i=\ov{1,d}$ such that $A_i=0$, then the statement
is proved by induction on $d$.

Let $\rank(A_j)=2$ for some $j=\ov{1,d}$. Conjugating, we
can suppose $A_j=J_2$. Lemma~\ref{lemma1} implies
that every $A_i$ is a strictly upper triangular matrix. Thus the
required is proved.

The only case we has not yet considered is $\rank(A_i)=1$,
$i=\ov{1,d}$ (see Lemma~\ref{lemma0}). Direct computations yields
$J_1BJ_1=\tr(J_1B)J_1$, $B\in M_3(K)$, and conjugating the equality we get
$A_iBA_i=\tr(A_iB)A_i$, $B\in M_3(K)$, $i=\ov{1,d}$. The latter
identity implies that if $V$ is a word in matrices $A_1,\ldots,A_d$
and $\deg_{A_j}(V)\geq2$ for some $j$, then $V=0$. Thus we has
shown that $\tr(U)=0$ for each word $U$ in $A_1,\ldots,A_d$.

Prove by induction on $\deg(V)$ that $\si_2(V)=0$, where $V$ is a
word in matrices $A_1,\ldots,A_d$. We have
$\si_2(UV)\equiv\tr(U^2V^2)$ (item~$4$ of Lemma~\ref{lemmaLR}), where
$U,V$ are words in $A_1,\ldots,A_d$. By induction hypothesis we
have $\si_2(U)=\si_2(V)=0$, which together with proved part of the
Lemma gives $\si_2(UV)=0$.

So the Lemma is proved.
\end{proof}

\begin{lemma}\label{lemma4}
Let $Q\subset R_{3,3}$ be the set
$$\begin{array}{l}
\si_k(X_i),\; k,i=\ov{1,3},\quad \tr(X_iX_j),\; 1\leq i<j \leq
3,\quad\tr(X_i^2X_j),\;
1\leq i\neq j \leq 3,\quad \\
\tr(X_i^2X_j^2),\; 1\leq i<j\leq 3,\quad \al \tr(X_1X_2X_3) +
\be\tr(X_1X_3X_2),
\end{array}$$
where $\al,\be\in K$ are non-zero. If $A_1,A_2,A_3\in M_3(K)$ are
such that $Q|_{A_1,A_2,A_3}=0$, then for each $r\in R_{3,3}^{+}$ we
have $r|_{A_1,A_2,A_3}=0$.
\end{lemma}
\begin{proof}If there is $j=\ov{1,3}$ such that $A_j=0$, then Lemma~\ref{lemma3}
concludes the proof.

If there is $j=\ov{1,3}$ such that $\rank(A_j)=2$, then conjugating
$A_1,A_2,A_3$ we can assume that $A_j=J_2$ (Lemma \ref{lemma0}).
Lemma~\ref{lemma1} implies that $A_1,A_2,A_3$ are strictly upper
triangular matrices. Thus the required is proved.

Let $\rank(A_i)=1$, $i=\ov{1,3}$. Conjugating, we can assume
$A=J_1$. Lemma~\ref{lemmaII} implies $A_1A_2=0$ or $A_2A_1=0$.
Thus $\tr(A_1A_2A_3)=0$ or $\tr(A_1A_3A_2)=0$. Lemma~\ref{lemma3}
concludes the prove.

By Lemma~\ref{lemma0}, all possibilities have been
considered.
\end{proof}

\begin{lemma}\label{lemma5}
Let $Q\subset R_{3,3}$ be the set
$$\begin{array}{l}
\si_k(X_i),\; k,i=\ov{1,3},\quad \tr(X_iX_j),\; 1\leq i<j \leq
3,\quad\tr(X_i^2X_j^2),\; 1\leq i<j\leq 3,\\
\tr(X_1^2X_3), \quad \tr(X_3^2X_2),\quad\tr(X_2^2X_1),\quad\\
\al_1\tr(X_1^2X_2)+\al_2\tr(X_2^2X_3)+\al_3\tr(X_3^2X_1), \\
\be_1 \tr(X_1X_2X_3) + \be_2\tr(X_1X_3X_2),\\
\end{array}$$
where $\al_1,\ldots,\be_2\in K$ are non-zero. If $A_1,A_2,A_3\in
M_3(K)$ are such that $Q|_{A_1,A_2,A_3}=0$, then for each $r\in
R_{3,3}^{+}$ we have $r|_{A_1,A_2,A_3}=0$.
\end{lemma}
\begin{proof} If there is $j=\ov{1,3}$ such that $A_j=0$, then Lemma~\ref{lemma3}
concludes the proof.

Let $\rank(A_1)=\rank(A_2)=2$. Applying Lemma~\ref{lemma1} to
matrices $A_1$, $A_3$ and $A_2$, $A_1$ we obtain that $\tr(A_3^2A_1)=0$
and $\tr(A_1^2A_2)=0$. Thus $\al_2\tr(A_2^2A_3)=0$. The statement
follows from Lemma~\ref{lemma4}.

Let $\rank(A_1)=\rank(A_2)=1$. Thus $A_1^2=A_2^2=0$
(Lemma~\ref{lemma0}). Hence
$\tr(A_1^2A_2)=\tr(A_2^2A_3)=0$, and thus $\al_3\tr(A_3^2A_1)=0$.
The statement follows from Lemma~\ref{lemma4}.

By symmetry and Lemma~\ref{lemma0}, all possibilities have
been considered.
\end{proof}

\begin{proofTh2}
In order to prove the Theorem it is sufficient to show that if
$A_1,A_2,A_3\in M_3(K)$ are such that $P|_{A_1,A_2,A_3}=0$, then
for each $r\in R_{3,3}^{+}$ we have $r|_{A_1,A_2,A_3}=0$ (see
Lemma~\ref{lemma_number} and the above mentioned Hilbert Theorem).
For convenience enumerate the equations:
\begin{eq}\label{eq1}
\tr(A_2A_3)=0.
\end{eq}\vspace{-5mm}
\begin{eq}\label{eq3}
\tr(A_1^2A_2)+\al_1\tr(A_2^2A_3)+\al_2\tr(A_3^2A_1)=0.
\end{eq}\vspace{-5mm}
\begin{eq}\label{eq4}
\tr(A_1^2A_3)-\be_1\tr(A_3^2A_2)=0.
\end{eq}\vspace{-5mm}
\begin{eq}\label{eq5}
\tr(A_1^2A_3)-\be_2\tr(A_2^2A_1)=0.
\end{eq}\vspace{-5mm}
\begin{eq}\label{eq2}
\tr(A_1A_2A_3)+\ga\tr(A_1A_3A_2)=0.
\end{eq}\vspace{-5mm}
\begin{eq}\label{eq8}
\tr(A_2^2A_3^2)=0.
\end{eq}

If there is $j=\ov{1,3}$ such that $A_j=0$, then Lemma~\ref{lemma3}
concludes the proof.

Let $\rank(A_j)=1$ for some $j=\ov{1,3}$. Thus $A_j^2=0$
(Lemma~\ref{lemma0}). The equations~\Ref{eq4},~\Ref{eq5} imply
that $\tr(A_1^2A_3)=\tr(A_3^2A_2)=\tr(A_2^2A_1)=0$.  The statement
follows from Lemma~\ref{lemma5}.

Let $\rank(A_i)=2$ for every $i=\ov{1,3}$. Conjugating, we can assume
that $A_1=J_2$ (Lemma~\ref{lemma0}). Let $A_2=T_2J_2T_2^{-1}$,
$A_3=T_3J_2T_3^{-1}$, where $T_2,T_3\in GL_3(K)$. Lemma~\ref{lemmaI} implies that there are
three possibilities for $A_2$ and three possibilities for $A_3$.
Instead of letters $t_i$ $(i=\ov{1,5})$ we will use letters $b_i$
for matrix $A_2$ and letters $c_i$ for matrix $A_3$
$(i=\ov{1,5})$. Denote by $(j,k)$ the case when  the $j^{\rm th}$
possibility is valid for $A_2$ and the $k^{\rm th}$ possibility is
valid for $A_3$.

Case $(1,1)$. Here $A_2$, $A_3$ are strictly upper triangular
matrices and the required is obvious.

Case $(1,2)$. The equation~\Ref{eq5} implies $1=0$. It is a
contradiction, thus the case is impossible.

Case $(1,3)$. The equation~\Ref{eq5} implies $1=0$; a contradiction.

Case $(2,1)$. The equation~\Ref{eq5} implies $b_3=0$.
Thus $\det(T_2)=0$; a contradiction.

Case $(2,2)$. The equation~\Ref{eq5} implies $b_2=b_4 b_5 -
\beta_2 b_3 (-c_2+c_4 c_5)$. Thus~\Ref{eq8} shows that $c_3=0$ or
$c_4=b_4$. In the first case $\det(T_3)=0$, and in the second case
the equation~\Ref{eq4} implies $1=0$;  a contradiction.

Case $(2,3)$. The equation~\Ref{eq5} gives $b_2=b_4 b_5 + \beta_2
b_3 c_2$, thus the equation~\Ref{eq8} implies $b_1=b_4^2+c_1-b_4 c_3$. The
equation~\Ref{eq4} implies $c_4=\be_1b_3c_2$. It follows
from~\Ref{eq3} that $\al_1+\be_1+\al_2\be_2=0$; a contradiction.

Case $(3,1)$. The equation~\Ref{eq5} implies $\be_2=0$; a contradiction.

Case $(3,2)$. The equation~\Ref{eq5} implies $b_4=\beta_2(c_2-c_4
c_5)$. Then the equation~\Ref{eq8} gives two possibilities:

\noindent $a)$ $c_3=0$. Thus the equation~\Ref{eq4} gives a contradiction.

\noindent $b)$ $b_1=c_1+(b_3-c_4)c_4$. Thus the equation~\Ref{eq4} gives
$c_2=(\be_1/\be_2) b_2c_3 + c_4c_5$. The equation~\Ref{eq3}
implies $\al_1+\be_1+\al_2\be_2=0$; a contradiction.

Case $(3,3)$. The equation~\Ref{eq8} implies $c_3=b_3$. Then the
equation~\Ref{eq4} gives a contradiction.
\end{proofTh2}

By the same way it was proved that if we slightly change h.s.o.p.-s of $R_{3,2}$,
$R_{4,2}$ for $p=0$ from~\cite{Teranishi86}, then we get
h.s.o.p.-s for arbitrary $p$:

\begin{prop}
1. The set $\{\si_k(X_i)$ ($i=1,2$,
$k=\ov{1,3}$),  $\tr(X_1X_2)$, $\tr(X_1^2X_2)$,  $\tr(X_1X_2^2)$,
$\tr(X_1^2X_2^2)\}$ is a h.s.o.p. of $R_{3,2}$.

2. The set $\{\si_k(X_i)$ ($i=1,2$,
$k=\ov{1,4}$),
  $\tr(X_1X_2)$,
  $\tr(X_1^2X_2)$,  $\tr(X_1X_2^2)$,
  $\tr(X_1^3X_2)$,  $\tr(X_1X_2^3)$, $\tr(X_1^2X_2^2)$,
  $\si_2(X_1X_2)$,
  $\si_2(X_1X_2^2)$,  $\si_2(X_1^2X_2)\}$ is a h.s.o.p. of $R_{4,2}$.
\end{prop}

\begin{center} { ACKNOWLEDGEMENTS} \end{center}

The author is grateful to A.N.Zubkov for helpful advices and
constant attention.

\end{document}